\documentclass[12pt,a4paper,reqno]{amsart}

\usepackage{graphicx}
\usepackage{amsfonts}
\usepackage{enumerate}
\usepackage[margin=3cm]{geometry}

\usepackage{overpic}

\newtheorem{theorem}{Theorem}
\newtheorem{proposition}[theorem]{Proposition}
\newtheorem{lemma}[theorem]{Lemma}
\newtheorem{corollary}[theorem]{Corollary}

\newcommand{\R}{\mathbb{R}}

\DeclareMathOperator{\e}{e}
\DeclareMathOperator{\ii}{i}

\title[Limit cycles from infinity in planar PWLS]{Limit cycles from a monodromic infinity \\ in planar piecewise linear systems}
\author[E. Freire]{Emilio Freire}
\address{Departamento Matem\'{a}tica Aplicada II, E.T.S. Ingenier\'{\i}a, Universidad de Sevilla, Camino de los Descubrimientos, 41092 Sevilla (Spain)}
\email{efrem@us.es}

\author[E. Ponce]{Enrique Ponce}
\address{Departamento Matem\'{a}tica Aplicada II, E.T.S. Ingenier\'{\i}a, Universidad de Sevilla, Camino de los Descubrimientos, 41092 Sevilla (Spain)}
\email{eponcem@us.es}

\author[J. Torregrosa]{Joan Torregrosa}
\address{Departament de Matem\`{a}tiques, Universitat Aut\`{o}noma de Barcelona, 08193 Bellaterra, Barcelona (Spain); Centre de Recerca Matem\`{a}tica, Campus de Bellaterra, 08193 Bellaterra, Barcelona (Spain)}
\email{torre@mat.uab.cat}

\author[F. Torres]{Francisco Torres}
\address{Departamento Matem\'{a}tica Aplicada II, E.T.S. Ingenier\'{\i}a, Universidad de Sevilla, Camino de los Descubrimientos, 41092 Sevilla (Spain)}
\email{ftorres@us.es}



\begin{document}

\begin{abstract}
	Planar piecewise linear systems with two linearity zones separated by a straight line and with a periodic orbit at infinity are considered. By using some changes of variables and parameters, a reduced canonical form with five parameters is obtained. Instead of the usual Bendixson transformation to work near infinity, a more direct approach is introduced by taking suitable coordinates for the crossing points of the possible periodic orbits with the separation straight line. The required computations to characterize the stability and bifurcations of the periodic orbit at infinity are much easier. It is shown that the Hopf bifurcation at infinity can have degeneracies of co-dimension three and, in particular, up to three limit cycles can bifurcate from the periodic orbit at infinity. This provides a new mechanism to explain the claimed maximum number of limit cycles in this family of systems. The centers at infinity classification together with the limit cycles bifurcating from them are also analyzed.
\end{abstract}

\maketitle

\section{Introduction and main results}\label{se:1}

The analysis of piecewise linear systems is nowadays an active field of research since certain modern devices are
well-modeled by this class of systems, see \cite{dBBCKN08}. Even for the simplest situation, as is the case of the aggregation of two planar linear systems, there are still unsolved problems; for instance, it is known that such discontinuous piecewise linear systems can have three limit cycles (see, for instance, \cite{BuzPesTor2013,FPT14,HuanYang12, LP12,LTT13}) but we still do not know if three is indeed the maximum number for them. 

\medskip

In the analysis of the dynamical richness in a differential system, an interesting source of knowledge comes from the study of all possible bifurcations that the system can undergo. Furthermore, it should be emphasized the relevance of including in such a bifurcation study the possible bifurcations from infinity. Here, we explore the information on the maximum number of limit cycles that can be obtained by studying the periodic orbit at infinity and its possible bifurcations in the mentioned family of planar discontinuous piecewise linear systems with two zones separated by a straight line. Bifurcations from infinity for planar piecewise linear differential systems have been analyzed before in \cite{LP99}, and more recently in \cite{GLN15}. In \cite{LP99} only continuous cases with two zones and three zones with symmetry were considered. For the two-zones case, only one bifurcating limit cycle was detected, according to the well-known fact that there can be only one limit cycle in the class of continuous planar piecewise linear differential systems with two zones separated by a straight line, see \cite{FPRdT98}. In \cite{GLN15}, the bifurcation from infinity is addressed for the case of discontinuous piecewise linear differential systems, by perturbing in a non-symmetric way the canonical continuous linear center $(\dot x,\dot y)=(-y,x)$, allowing for different linear perturbations in the half-planes $y<0$ and $y>0$. Again, only one bifurcating limit cycle was obtained. In both cases, the technical procedure for the analysis takes advantage of the Bendixson transformation. Such a technique is also used in the recent work \cite{BMT20}, where a different family of piecewise linear systems with a symmetric 4-star structure is analyzed. There, after a rather involved computational work, the quoted authors show that up to five limit cycles can bifurcate from infinity.

\medskip

In this paper we propose an alternative and more direct way to work near infinity so that it is possible to build without excessive computational effort a Poincar\'e-like return map that allows to characterize in a complete way the periodic orbit at infinity, discriminating several cases where such orbits belong to a period annulus and the cases where such a periodic orbit behaves like a weak-focus. It is shown that the maximum order for the infinity being a weak-focus is three, so we also show that up to three limit cycles bifurcate from infinity. This achievement is very relevant because we provide a new mechanism to generate the supposedly maximum number of limit cycles for the family via just a local analysis. In fact, the phenomenon could be termed a \emph{degenerate Hopf bifurcation} at infinity.

\medskip

Effectively, to explain the existence of three limit cycles on discontinuous piecewise linear differential systems with two zones separated by a straight line, different mechanisms have been proposed. In \cite{BM13} authors propose a degenerate boundary equilibrium bifurcation of non-smooth Hopf-like type to pass from a configuration without periodic orbits to another with three limit cycles. In \cite{BuzPesTor2013} they appear perturbing the harmonic oscillator via the piecewise averaging technique of high-order. In \cite{CLNT20} they are obtained by perturbing a global center with a different piecewise linear system in each zone. In \cite{FPT14}, starting from a situation possessing one limit cycle coexisting with a boundary focus, two new limit cycles are obtained by taking advantage of the boundary focus unfolding.

\medskip

We emphasize that our alternative formulation of the closing equations, whose local analysis near infinity is the subject of this work, might be useful to get the upper bound for the total number of limit cycles in the family of systems under study. This should be the subject of future work. The existence of such upper bound has been proved only for some special non-generic classes; see for instance \cite{LZ18}, where focus type dynamics is not allowed.

\medskip

We start our analysis by assuming without loss of generality that the two regions in the phase plane are the left and right half-planes,
\[
S_L=\{(x,y)\in\R^2: x < 0\},    \quad  S_R=\{(x,y)\in\R^2: x > 0\},
\]
separated by the straight line $\Sigma = \{(x,y)\in\R^2:x = 0\}$.
The systems to be studied become
\[
\mathbf{\dot{x}} = 
\begin{cases}
      A_L\mathbf{x}+\mathbf{b}_L,& \text{ if } x\in S_L \cup \Sigma,\\
      A_R\mathbf{x}+\mathbf{b}_R,& \text{ if } x\in S_R,
\end{cases}
\]
where $\mathbf{x}=(x,y)^{T}\in\mathbb{R}^{2}$, $A_{L}=(a_{ij}^{L})$ and $A_{R}=(a_{ij}^{R})$ are $2\times2$ constant matrices with real coefficients and $\mathbf{b}_L=(b_{1}^{L},b_{2}^{L})^{T},$ $\mathbf{b}_{R}=(b_{1}^{R},b_{2}^{R})^{T}$ are constant vectors in $\mathbb{R}^{2}.$ Over the separation line $\Sigma$ we define the vector field using the Filippov convention, see \cite{Fil1988}. As we have commented before, we are interested only in solutions near the periodic orbit at infinity.  Under the generic condition $a_{12}^L a_{12}^R > 0$, orbits sufficiently far from the origin cross the discontinuity line, allowing the existence of periodic orbits living in both half-planes. These kind of orbits are usually called of crossing type. Under such a generic condition, the points in $\Sigma$ that cannot be part of a crossing orbit, i.e. sliding or escaping ones, where Filippov convention is necessary, form a bounded set.

\medskip

Therefore, by using a similar approach to the one followed in getting Proposition~3.1 of
\cite{FPT12} and denoting $\Lambda\in\{L,R\}$, we obtain the new canonical form 
\begin{equation}\label{eq:1}
        \mathbf{\dot{x}} =
        \left(\begin{array}{cr}
            T_\Lambda & -1\\
            D_\Lambda & 0
        \end{array}\right)
        \mathbf{x} -
        \left(\begin{array}{c}
            b_\Lambda\\
            a_\Lambda
        \end{array}\right)
        \text{ if } \mathbf{x} \in S_\Lambda,
\end{equation}
with $T_{\Lambda} = \operatorname{tr}(A_{\Lambda})$ and $D_{\Lambda} = \det(A_{\Lambda})$ are the linear invariants in each zone and $b_L=b$ and $b_R=-b$.

\medskip

The canonical form \eqref{eq:1} has seven parameters; apart from the mentioned linear invariants, we find two parameters $a_{L}, a_{R},$ related to the position of equilibria and a parameter $b$ which is responsible for the existence of a sliding set. In fact, the sliding set is the segment joining the points $(0,-b)$ and  $(0,b)$, see \cite{FPT12} for more details. These two endpoints are the tangency points of system \eqref{eq:1}, so that the sliding segment becomes attractive for $b<0$ and repulsive for $b>0$, shrinking to the origin when $b=0$. By computing the sign of $\ddot{x}$ at the tangency points, we obtain
\[
\ddot{x}|_{(x,y)=(0,-b)}=a_{L},\qquad \ddot{x}|_{(x,y)=(0,b)}=a_{R},
\]
so that the left (right) tangency is called visible if $a_{L} < 0$ ($a_{R} > 0$), being invisible if $a_{L} > 0$ ($a_{R} < 0$), see again \cite{FPT12}. Thus, the $a_{\Lambda}$ parameters are related with the location of the equilibria and determine the visibility of the tangencies; when some of them vanish then we have a boundary equilibrium point, see \cite{KRG03,PPT11}.

\medskip

Our main hypothesis will be the monodromy of the point at infinity, that is the existence of a periodic orbit at infinity, which requires to have no equilibrium points there. This is equivalent to ask for having dynamics of focus type in both regions, see \cite{LP99}, namely $T_{\Lambda}^2-4D_{\Lambda} < 0.$ We note that under the above conditions both determinants are positive. 

\medskip

As a preliminary result, necessary to state our main theorems, we introduce a new (symmetric) canonical form, that it will be used in our approach to study limit cycles bifurcating from infinity. 

\begin{proposition}\label{prop:1}
Under the hypotheses $T_{\Lambda}^2 - 4D_{\Lambda} < 0$
(both dynamics are of focus type), system \eqref{eq:1} is
topologically equivalent to system
\begin{equation}\label{eq:2}
\begin{array}{ll}
    \left\{\begin{array}{l}
        \dot{x}=2\gamma_{L} x - y-b, \\ 
        \dot{y}=(1 + \gamma_{L}^{2})x -\alpha_{L},
    \end{array}\right.
      \text{if }x \leq 0,  \text{ and }&
    \left\{\begin{array}{l}
        \dot{x}=2\gamma_{R} x - y +b,\\ 
        \dot{y}=(1 + \gamma_{R}^{2})x-\alpha_{R},
    \end{array}\right.
    \text{if }x > 0,
\end{array}
\end{equation}
where, for each zone $\Lambda\in\{L,R\}$, we introduce the new parameters 
\begin{equation}\label{eq:3} 
\gamma_{\Lambda}= \dfrac{\sigma_{\Lambda}}{\omega_{\Lambda}},\quad \alpha_{\Lambda} =\dfrac{a_{\Lambda}}{\omega_{\Lambda}},
\end{equation}
with $2\sigma_{\Lambda} =T_{\Lambda}$ and $\omega_{\Lambda}> 0$ is such that $4\omega_{\Lambda}^2 = 4D_{\Lambda} -T_{\Lambda}^2$.
\end{proposition}
We have reduced by two the number of parameters in \eqref{eq:1} but, what is more important, we make patent the intrinsic features of the dynamics in each region. Effectively, the eigenvalues for the foci in \eqref{eq:2} are now $\gamma_{\Lambda}\pm \ii$, so that the natural frequencies are scaled to $1$ for both dynamics and the dynamical expansion or contraction for each focus depends only on the coefficients $\gamma_{\Lambda},$  being again $\Lambda\in\{L,R\}.$

\medskip
A direct consequence of the above proposition is the characterization of the continuity for system \eqref{eq:2}, leading to a new reduced canonical form with only three free parameters.
\begin{corollary}\label{coro:2}
System \eqref{eq:2} becomes continuous if and only if $b=0$ and $\alpha_L=\alpha_R.$ Hence, it writes as 
\begin{equation*}
\begin{array}{ll}
\left\{\begin{array}{l}
\dot{x}=2\gamma_{L} x - y, \\ 
\dot{y}=(1 + \gamma_{L}^{2})x -\alpha,
\end{array}\right.
\text{if }x \leq 0,  \text{ and }&
\left\{\begin{array}{l}
\dot{x}=2\gamma_{R} x - y,\\ 
\dot{y}=(1 + \gamma_{R}^{2})x-\alpha,
\end{array}\right.
\text{if }x \ge 0,
\end{array}
\end{equation*}
being $\alpha$ the common value for the non-homogeneous terms.
\end{corollary}

\medskip

Before stating our first result about the characterization of the existence of a period annulus at infinity for system \eqref{eq:2}, we recall the notion of time-reversibility with respect to straight lines. Whenever a planar system is invariant under the change $(x,y,\tau)\mapsto (x,-y,-\tau)$ or $(x,y,\tau)\mapsto (-x,y,-\tau),$ we say that is time-reversible with respect to the $x$-axis or $y$-axis, respectively.

\begin{theorem} \label{thm:3}
System \eqref{eq:2} has a center (period annulus) at infinity if and only if it is time-reversible with respect to $y=0$ or $x=0.$

The centers have a time-reversibility with respect to $y=0$  if and only if $b=0$ and  $\gamma_L=\gamma_R=0.$   The centers have a time-reversibility with respect to $x=0$  if and only if $b=0,$ $\gamma_L=-\gamma_R\ne0$ and either $\alpha_L=\alpha_R=0$ or $\alpha_L=-\alpha_R\ne0$.
\end{theorem}

The proof of this result is a direct consequence of a more complete one, where we detail also the global qualitative behavior, see Theorem~\ref{thm:5}. Its proof is based upon the derivation of an adequate Poincar\'e return map that allows to study a neighborhood of infinity as if it were a standard monodromic point. This idea has been used many times, see \cite{BMT20,CHYH18, GLN15,LLiuYu2017,LP99}, by resorting to the Bendixson transformation; the computations become rather involved since, as shown later, to `determine' the stability of the equilibrium point one needs to compute derivatives of high-order of such a Poincar\'e map. Here, we exploit an alternative and more convenient approach, by introducing a new suitable coordinate $u_0$ associated to one of the two intersection points of the periodic orbit with the separation straight line,  the value $u_0=0$ corresponding to the periodic orbit at infinity. Thus, we are able to compute much more easily a displacement map in the form
\begin{equation}\label{eq:4}
\Delta(u_0)=\Delta_1 u_0+\Delta_2 u_0^2+\Delta_3 u_0^3+\Delta_4 u_0^4+\cdots,
\end{equation}
for $u_0>0$ and small, such that its positive zeros have a one-to-one correspondence with periodic orbits near infinity. Furthermore, the coefficients $\Delta_i$ determine the stability and the weak-focus or center character of the periodic orbit at infinity. In particular, when there exists a period annulus at infinity then we can say that the infinity behaves like a center and all the above coefficients vanish. The reciprocal statement is also true. When $\Delta_1=0$ the periodic orbit at infinity is non-hyperbolic and then, provided that the first non-vanishing coefficient in the above expansion is $\Delta_i$, we say that the periodic orbit at infinity behaves like a weak-focus of order $i-1$. Thus, our second main result assures that the maximum order of the periodic orbit at infinity when it behaves like a weak-focus is three, see Section~\ref{se:5} for a proof.

\begin{theorem}\label{thm:4}
	For system \eqref{eq:2}, the periodic orbit at infinity is hyperbolic and stable (unstable) whenever $\gamma_L+\gamma_R>0$ $(\gamma_L+\gamma_R<0)$. When $\gamma_L+\gamma_R=0$ the periodic orbit at infinity is non-hyperbolic so that it behaves like a weak-focus or a center. The possible weak-focus orders are only $1,2,$ and $3$ and there exist perturbations such that the system exhibits $1,2,$ and $3$ limit cycles of big amplitude, respectively.
\end{theorem}

We notice that it is the first time that in this family of systems the associated Hopf bifurcation is shown to be up to of co-dimension three; furthermore, it is proved that up to three limit cycles can bifurcate from infinity. 

\medskip

This paper is structured as follows. Section~\ref{se:2} presents apart from some properties satisfied by system \eqref{eq:2}, other canonical forms associated to system~\eqref{eq:1}. Proposition~\ref{prop:1} is proved also here. How are the half-return maps near infinity and the computation of the coefficients of the displacement function \eqref{eq:4} are done in Section~\ref{se:3}. The center characterization result, Theorem~\ref{thm:3}, is shown with more details through Theorem~\ref{thm:5} in Section~\ref{se:4}. In Section~\ref{se:5} we get the different possible weak-focus orders and the corresponding limit cycles bifurcation that the periodic orbit at infinity can have, see Theorems~\ref{thm:7} and \ref{thm:8}, jointly leading to Theorem~\ref{thm:4}. The limit cycles near infinity that bifurcate from the centers are studied in Section~\ref{se:6}, see Propositions~\ref{prop:10}, \ref{prop:11}, and \ref{prop:12}. Finally, Section~\ref{se:7} deals with an explicit example where the three limit cycles that bifurcate from infinity are numerically shown.

\section{About the canonical forms}\label{se:2}
In this paper we basically work with the canonical form \eqref{eq:2} but some other equivalent forms are also interesting. First we introduce some notation and properties on the equilibrium points of \eqref{eq:2} which are of focus type:
\begin{equation}\label{eq:5}
	(x_L,y_L) =(x_L,2\gamma_Lx_L-b)=\left( \dfrac{\alpha_{L}}{1 + \gamma_{L}^{2}},\dfrac{2\alpha_{L}\gamma_{L}}{1 + \gamma_{L}^{2}}-b\right)
\end{equation}
and
\begin{equation}\label{eq:6}
(x_R,y_R)=(x_R,2\gamma_Rx_R+b)=\left( \dfrac{\alpha_{R}}{1 + \gamma_{R}^{2}},\dfrac{2\alpha_{R}\gamma_{R}}{1 + \gamma_{R}^{2}}+b\right).
\end{equation}
As the vector fields in \eqref{eq:2} are linear, it is clear that the equilibrium points are stable (unstable) for $\gamma_{\Lambda} < 0$ $(\gamma_{\Lambda} >0).$ When $\gamma_{\Lambda} = 0$, we have linear centers. Such equilibria will be real when $\alpha_{L}<0$ or $\alpha_{R}>0$, boundary equilibria for $\alpha_{\Lambda}=0$, and virtual ones when $\alpha_{L}>0$ or $\alpha_{R}<0$. 

In terms of the equilibrium coordinates \eqref{eq:5} and \eqref{eq:6}, we can rewrite system  \eqref{eq:2} as follows,
\begin{equation}\label{eq:7}
	\left\{\begin{array}{l}
	\dot{x}=2\gamma_{L} (x-x_L) - (y-y_L), \\ 
	\dot{y}=(1 + \gamma_{L}^{2})(x -x_L),
	\end{array}\right.
	\quad
	\left\{\begin{array}{l}
	\dot{x}=2\gamma_{R} (x-x_R)  - (y-y_R), \\ 
	\dot{y}=(1 + \gamma_{R}^{2})(x -x_R),
	\end{array}\right.  
\end{equation}
for $x\leq0$ and $x>0$, respectively, and note that now the family is described with $6$ parameters, one more than in \eqref{eq:2}. The parameter $b$ has been rewritten after introducing the equilibrium ordinates $y_\Lambda$ and the parameters $\alpha_{\Lambda}$ have been substituted by the corresponding equilibrium abscissas  $x_{\Lambda}$, for $\Lambda\in\{L,R\}.$ We have that 
\begin{equation}\label{eq:8}
	\alpha_{\Lambda}=(1+\gamma_{\Lambda}^2)x_{\Lambda}, \quad b=-(y_L-2\gamma_L x_L)=y_R-2\gamma_R x_R,
\end{equation}
and the last equality gives the condition to be fulfilled by the six parameters in a system \eqref{eq:7} to be equivalent to a system in the form \eqref{eq:2}. However, every system \eqref{eq:7} not fulfilling the last equality in \eqref{eq:8} can be rewritten with a simple translation in the variable $y$ in another equivalent system, already satisfying the mentioned condition. In fact, such a condition amounts to have the origin in the middle of the sliding set, which is a segment in the $y$-axis.

\medskip

Now we can prove our first main result.
\begin{proof}[Proof of Proposition~\ref{prop:1}]
	Under the hypotheses, if we define $\omega_R> 0$ such that $\omega_R^2 = D_{R} - T_{R}^2/4$ and $\sigma_R = T_{R}/2$,
	the eigenvalues of the matrix ruling the dynamics on the half-plane $S_R$ in \eqref{eq:1} are $\sigma_R\pm \ii\omega_R$. Note that  $D_R=\sigma_R^2+\omega_R^2.$
	We make first the change $X = \omega_R x$, $Y = y$, $\tau = \omega_R t$ for the variables in $S_R$, without altering variables and time on the half-plane $S_L$.
	Note that we do not change the coordinate $y$, so that periodic orbits using both half-planes are preserved. Then, for $X>0$ we have
	\[
	\begin{aligned}
	\dfrac{dX}{d\tau} & = \dfrac{1}{\omega_R}\dfrac{dX}{dt} = \dfrac{dx}{dt} = \dfrac{T_{R}X}{\omega_R} - Y+b,\\
	\dfrac{dY}{d\tau} & = \dfrac{1}{\omega_R}\dfrac{dY}{dt} = \dfrac{1}{\omega_R}\dfrac{dy}{dt} = \dfrac{1}{\omega_R}\left(D_{R} \dfrac{X}{\omega_R} - a_{R}\right) = \dfrac{D_{R}}{\omega_R^2}X - \dfrac{a_{R}}{\omega_R}.
	\end{aligned}
	\]
	Introducing the parameter $\gamma_R$ as in \eqref{eq:3}, we see that 
	\[
	\dfrac{T_{R}}{\omega_R} = 2\gamma_R,\quad \dfrac{D_{R}}{\omega_R^2} = \gamma_R^2 + 1.
	\]
	Thus, the new vector field for the right half-plane is as given in the statement with $\alpha_R$ as in \eqref{eq:3}. Doing the analog transformation for the left half-plane we get a similar result, and the proposition is proved.   
\end{proof}

As an intermediate option between the two forms \eqref{eq:2} and \eqref{eq:7}, we can also use the 5-parameter formulation
\begin{equation}\label{eq:9}
	\left\{\begin{array}{l}
	\dot{x}=2\gamma_{L} x - y-b, \\ 
	\dot{y}=(1 + \gamma_{L}^{2})(x -x_{L}),
	\end{array}\right.
	\quad
	\left\{\begin{array}{l}
	\dot{x}=2\gamma_{R} x - y +b,\\ 
	\dot{y}=(1 + \gamma_{R}^{2})(x-x_{R}),
	\end{array}\right.
\end{equation}
for $x\leq0$ and $x>0$, respectively. Regarding the form \eqref{eq:9}, the system becomes invariant under the transformations
\begin{align}
			(x,y,\tau,\gamma_L,x_L,b,\gamma_R,x_R)&\mapsto (-x,y,-\tau,-\gamma_R,-x_R,-b,-\gamma_L,-x_L),\label{eq:10}\\
			(x,y,\tau,\gamma_L,x_L,b,\gamma_R,x_R)&\mapsto (x,-y,-\tau,-\gamma_L,x_L,-b,-\gamma_R,x_R),\notag
\end{align}
and	its composition
\[
(x,y,\tau,\gamma_L,x_L,b,\gamma_R,x_R)\mapsto (-x,-y,\tau,\gamma_R,x_R,b,\gamma_L,x_L).
\]
The new time $\tau$ has been introduced in the proof of Proposition~\ref{prop:1}. 

These properties are useful to reduce the number of configurations to be considered for the analysis of the family. In fact, the parameter $b$ is \emph{modal} in the sense that by means of a homogeneous scaling in the variables $(x,y)$, which also implies to scale accordingly the parameters $(x_L,x_R)$ in \eqref{eq:9}, only the three cases $b=1$ (repulsive sliding segment), $b=0$ (sewing case), and $b=-1$ (attractive sliding segment) should be considered. We will not take advantage of this last observation as we are interested in a bifurcation approach to our problem, which requires as much as possible to modify the parameters in a continuous way.

\section{Half-return maps near infinity}\label{se:3}

The periodic orbits of system \eqref{eq:2} near infinity can be determined from the half-return maps, $L$ and $R$, near infinity on each side. We take a point $(0,y_0)$ with $y_0>0$ as initial point of an orbit for the left system, and integrate the solution forward in time up to arrive again, after approximately a half tour around the focus at $(x_L,y_L)$, to the $y$-axis. The existence of an arrival point of the form $(0,y_1)$ with $y_1<0$ is guaranteed as long as $y_0$ is chosen sufficiently big. Similarly, for the right side, we just integrate the right system backward in time, also starting at the point $(0,y_0)$ and arriving to a point $(0,y_2)$ with $y_2<0.$ These intersection points in the negative vertical axis define the half-return maps $L(u_0)=1/y_1$ and $R(u_0)=u_2=1/y_2$ being $u_0=1/y_0,$ for $u_0>0$ and small. Then, we can define the displacement map 
\begin{equation}\label{eq:11}
\Delta(u_0)=L(u_0)-R(u_0).
\end{equation}
We will see in the following that
\[
\begin{aligned}
L(u_0)&= L_1 u_0+L_2 u_0^2+L_3 u_0^3+L_4 u_0^4+\cdots,\\
R(u_0)&= R_1 u_0+R_2 u_0^2+R_3 u_0^3+R_4 u_0^4+\cdots.\\
\end{aligned}
\]
It should be clear that the positive zeros of the difference function \eqref{eq:11} correspond with periodic orbits near the periodic orbit at infinity. Thus the periodic orbit at infinity will be stable (unstable) when for $u_0>0$ and small we have $y_1-y_2<0$ $(y_1-y_2>0)$. Furthermore, we will see that when $L_1-R_1=L'(0)-R'(0)$ is non-vanishing, the periodic orbit at infinity will be hyperbolic and its sign determines its stability. More concretely, when $L_1-R_1>0 \, (L_1-R_1<0)$ the infinity of \eqref{eq:2} is stable (unstable). 

\medskip

Alternatively, the orbit passing through the point $(0,y_0)$ with $y_0=1/u_0>0$ can be thought as the orbit that terminates at the point $(0,y_1)$ with $y_1=1/u_1<0$ after a complete turn starting at the point $(0,y_2)$ with $y_2=1/u_2<0$, defining a pseudo-Poincar\'e return map $u_1=\Pi(u_2)=L(R^{-1}(u_2))$. The first derivative is 
\[
\Pi'(u_2)=L'(R^{-1}(u_2))\dfrac{1}{R'(R^{-1}(u_2))},
\]
so that for $u_2=0$ we have $R^{-1}(0)=0$ and then $\Pi'(0)$ reduces to $L_1/R_1$, being this quotient the unity when $\gamma_L+\gamma_R=0$. 

\medskip

Thus, these half-return maps will allow to determine the stability of the periodic orbit at infinity and the birth of other periodic orbits from infinity in a degenerate Hopf type bifurcation. The Taylor series of the displacement map \eqref{eq:11} has all monomials, contrary to what happens in the analytical case in which it is shown that the first non-vanishing coefficient always corresponds to an odd exponent, see \cite{AndLeoGorMai1973}. Moreover, in piecewise differential systems, the return map near a monodromic equilibrium point has a constant term due to the existence of a sliding segment, see \cite{FrePonTor2014}. However, here $\Delta(0)=0$ since the infinity remains invariant under any perturbation.

\medskip

Let us start by considering the left side. Thanks to Proposition~\ref{prop:1}, we already can assume that 
\begin{equation*}
   A_L=
        \left(\begin{array}{cr}
            2\gamma_{L} & -1\\
            1+ \gamma_{L}^2& 0
        \end{array}\right),
 \end{equation*}
and, instead of writing the solution of the differential system starting at the point $(0,y_0)$, we can take advantage of the fact that the exponential matrix $\exp(A_L \tau_L)$ is a fundamental matrix for the corresponding variational system, where $\tau_L$ is the time elapsed between two points of a given orbit. Thus, we have the following relation between the vector field at the arrival point and the vector field at the starting point,
\begin{equation*}
\left(\begin{array}{c}
            -y_1 -b\\
             -\alpha_{L}
        \end{array}\right)
        =\e^{A_L\tau_L} 
        \left(\begin{array}{c}
            -y_0 -b\\
             -\alpha_{L}
        \end{array}\right),
\end{equation*}
or equivalently,
\begin{equation}\label{eq:12}
\left(\begin{array}{c}
            y_1 +b\\
             \alpha_{L}
        \end{array}\right)
        -\e^{A_L\tau_L} 
        \left(\begin{array}{c}
            y_0 +b\\
             \alpha_{L}
        \end{array}\right)=\left(\begin{array}{c}
            0\\
            0
        \end{array}\right).
\end{equation}
         
In order to work near infinity, we introduce new suitable variables that allow us to work as if we were working near an ordinary equilibrium point, without needing to transform the differential equation (as it happens with the Bendixson transformation, see \cite{LP99,GLN15}). The key point is to introduce a suitable change of variables once written the closing equations that determine the periodic orbits of the system; recently, the same idea has been successfully extended to 3D systems in \cite{FPRVA20}. First, as the time $\tau_L$ should be near $\pi$ when $y_0$ is very big, it seems natural to take a new time variable $s_L=\tau_L-\pi$ but, what is more relevant, we also introduce in equation \eqref{eq:12} the new variables
\[
u_0=y_0^{-1},\quad u_1=y_1^{-1}
\]
so that we get, after some standard manipulations, the equation
\[
\left(\begin{array}{c}
            u_0 +bu_0u_1\\
             \alpha_{L}u_0u_1
        \end{array}\right)
        -\e^{A_L(\pi+s_L)} 
        \left(\begin{array}{c}
            u_1 +bu_0u_1\\
             \alpha_{L}u_0u_1
        \end{array}\right)=\left(\begin{array}{c}
            0\\
            0
        \end{array}\right),
\]
where $u_0>0$,  $u_1<0,$ and $s_L$ are assumed to be small enough. Thus, we want to solve the above equation in a neighborhood of the point $(u_0,u_1,s_L)=(0,0,0)$, which turns out to be an equilibrium point. 
    
It is convenient to split the exponential matrix into the product of two matrices, by noting that $\exp(A_L \pi) =-\exp(\gamma_L\pi)I$. After multiplying the last equation by the scalar
\[
\e_L^{-}=\e^{-\gamma_L\pi},
\]
we get
\begin{equation}\label{eq:13}
\e_L^{-}\left(\begin{array}{c}
            u_0 +bu_0u_1\\
             \alpha_{L}u_0u_1
        \end{array}\right)
        + \e^{A_Ls_L} 
        \left(\begin{array}{c}
            u_1 +bu_0u_1\\
             \alpha_{L}u_0u_1
        \end{array}\right)=\left(\begin{array}{c}
            0\\
            0
        \end{array}\right).
\end{equation}
Now, to desingularize equation \eqref{eq:13}, it is enough to remove from the second component the trivial factor $u_1$, and write the equation
    
\begin{equation}\label{eq:14}
\e_L^{-}\left(\begin{array}{c}
            u_0 +bu_0u_1\\
             \alpha_{L}u_0
        \end{array}\right)
        +\left(\begin{array}{cc}
            u_1 & 0\\
             0 & 1
        \end{array}\right)\e^{A_Ls_L} 
        \left(\begin{array}{c}
            1 +bu_0\\
             \alpha_{L}u_0
        \end{array}\right)=\left(\begin{array}{c}
            0\\
            0
        \end{array}\right),
\end{equation}
whose Jacobian with respect to $(u_0,u_1,s_L)$ at $(0,0,0)$ is the full-rank matrix
\[
\left(\begin{array}{ccc}
         \e_L^{-} & 1 & 0\\
       (1+\e_L^{-})\alpha_L & 0 & 1+\gamma_L^2
    \end{array}\right).
\]
It is possible now to apply the Implicit Function Theorem at the point $(u_0,u_1,s_L)=(0,0,0)$, to assure the existence of unique expansions for $u_1=L(u_0)$ and $s_L=\beta(u_0)$ in terms of $u_0$, namely
\[
\begin{aligned}
u_1&= L(u_0)= L_1 u_0+L_2 u_0^2+L_3 u_0^3+L_4 u_0^4+\cdots,\\
s_L&= \beta(u_0)=\beta_1 u_0+\beta_2 u_0^2+\beta_3 u_0^3+\beta_4 u_0^4+\cdots.
\end{aligned}
\]
Computations can be done in a degree by degree manner regarding the powers of $u_0$, taking into account that
\[
\begin{aligned}
 \e^{A_Ls_L} &=I+\beta_1A_Lu_0+\left(\beta_2A_L+\frac{\beta_1^2}{2}A_L^2\right)u_0^2+\left(\beta_3A_L+\beta_1\beta_2A_L^2+\frac{\beta_1^3}{6}A_L^3\right)u_0^3\\ 
 &\phantom{=} +\left(\beta_4A_L+\frac{\beta_2^2+2\beta_1\beta_3}{2}A_L^2+\frac{\beta_1^2\beta_2}{2}A_L^3+\frac{\beta_1^4}{24}A_L^4\right)u_0^4+\cdots,
\end{aligned}
\]
and separating the left hand side terms of \eqref{eq:14} in the form
\[
\e_L^{-}\left[u_0\left(\begin{array}{c}
            1 \\
             \alpha_{L}
        \end{array}\right)+u_0u_1\left(\begin{array}{c}
            b \\
            0
        \end{array}\right)\right]
        +\left(\begin{array}{cc}
            u_1 & 0\\
             0 & 1
        \end{array}\right)\e^{A_Ls_L} 
         \left[\left(\begin{array}{c}
            1 \\
           0
        \end{array}\right)+u_0\left(\begin{array}{c}
            b\\
            \alpha_{L}
        \end{array}\right)\right].
\]
For instance, the vanishing of the first degree terms in \eqref{eq:14} gives
\[
            \e_L^{-}\left(\begin{array}{c}
            1\\
             \alpha_{L}
        \end{array}\right)
        +\left(\begin{array}{c}
           L_1\\
           \alpha_L+(1+\gamma_L^2)\beta_1
        \end{array}\right)=\left(\begin{array}{c}
            0\\
            0
        \end{array}\right),
\]
so that
\begin{equation}\label{eq:15}
L_1=-\e_L^{-},\quad\beta_1=-\frac{1+\e_L^{-}}{1+\gamma_L^2}\alpha_L=-(1+\e_L^{-})x_L.
\end{equation}
Regarding second order terms, we have
\[
            \e_L^{-}\left(\begin{array}{c}
            bL_1\\
            0
        \end{array}\right)
        +\left(\begin{array}{c}
           L_2+2\gamma_L\beta_1L_1+bL_1\\
           (1+\gamma_L^2)(\beta_2+b\beta_1+\gamma_L\beta_1^2)
        \end{array}\right)=\left(\begin{array}{c}
            0\\
            0
        \end{array}\right),
\]
so that
\begin{equation}\label{eq:16}
L_2=\e_L^{-}(1+\e_L^{-})(b-2\gamma_Lx_L)=-\e_L^{-}(1+\e_L^{-})y_L,\quad\beta_2=-b\beta_1-\gamma_L\beta_1^2,
\end{equation}
and so on. We have also obtained $L_3$, $\beta_3$, $L_4,$ and  $\beta_4$. Here, we write the final expressions for $L_3$ and $L_4$, which will be needed later for the analysis, namely
\begin{equation}\label{eq:17}
\begin{aligned}
L_3&=-\e_L^{-} (1 + \e_L^{-}) \left((1 + \gamma_L^2) \frac{\e_L^{-}-1}{2} x_L^2 +(1 +  \e_L^{-}) y_L^2\right),\\
L_4&=-\e_L^{-} (1 + \e_L^{-}) Q_L,
\end{aligned}
\end{equation}
where
\[
  Q_L=
 (1 + \gamma_L^2)\left(\frac{2\gamma_L(1 - 
      \e_L^{-} + (\e_L^{-})^2)}{3} x_L^3 + 
   \frac{(\e_L^{-}-1)(2 \e_L^{-}+ 3 )}{2}  x_L^2 y_L\right) + 
    (1 + \e_L^{-})^2 y_L^3.
\]

\medskip

The procedure can be repeated step by step for the right half-plane, starting from equation 
\[
\left(\begin{array}{c}  -y_0 +b\\ -\alpha_{R} \end{array}\right)
        =\e^{A_R\tau_R}  \left(\begin{array}{c} -y_1 +b\\ -\alpha_{R} \end{array}\right),
\]
just, by the symmetry of our model, changing $(y_0,y_1,b,L)$ by $(y_1,y_0,-b,R),$ respectively. Because now the point $(0,y_1)$ is the initial point and $(0,y_0)$ is the final point, the parameter $b$ has now a plus sign, and all the subscripts are $R$ instead of $L$. We introduce the equivalent values $s_R=\tau_R-\pi$ and 
\[
\e_R^{+}=\e^{\gamma_R\pi}.
\]
Notice that we want to obtain, for the orbit in the right half-plane that  arrives at $(0,1/u_0)$ starting from the point $(0,1/u_1)$ with $u_0>0$ and $u_1<0$, being both small enough, the expansion 
\[
u_1 =R(u_0)=R_1 u_0+R_2 u_0^2 +R_3 u_0^3 + R_4 u_0^4+\cdots.
\]
We obtain
\begin{equation}\label{eq:18}
\begin{aligned}R_1&=-\e_R^{+},\\
R_2&=-\e_R^{+} (1 + \e_R^{+}) (2 \gamma_R x_R+b)=-\e_R^{+} (1 + \e_R^{+}) y_R,\\
R_3&=-\e_R^{+} (1 + \e_R^{+}) \left( (1 + \gamma_R^2) (\e_R^{+}-1) x_R^2 /2+(1 +  \e_R^{+}) y_R^2\right),\\
R_4&=-\e_R^{+} (1 + \e_R^{+}) Q_R
\end{aligned}
\end{equation}
where 
\[
Q_R= (1 + \gamma_R^2)\left(\frac{2\gamma_R(1 -  \e_R^{+} + (\e_R^{+})^2)}{3} x_R^3 + \frac{(\e_R^{+}-1) (2 \e_R^{+}+ 3 )}{2} x_R^2 y_R\right) +   (1 + \e_R^{+})^2 y_R^3.
\]
These coefficients could be directly derived from $L_i,$ for $i=1,\ldots,4$ by using the transformation \eqref{eq:10} restricted to the parameters space, namely
\[(\gamma_L,x_L,y_L,b,\gamma_R,x_R,y_R)\mapsto (-\gamma_R,-x_R,y_R,-b,-\gamma_L,-x_L,y_L).\]

From \eqref{eq:15}, \eqref{eq:16}, \eqref{eq:17}, and \eqref{eq:18} we can write the first terms in the Taylor series of the displacement map \eqref{eq:11}. We will see in the next sections that we only need these four coefficients to characterize the centers and the maximum weak-focus order at infinity.
 
\section{The centers characterization}\label{se:4}
This section is devoted to prove our main result Theorem~\ref{thm:3} that characterizes when \eqref{eq:2} has a center at infinity. In fact, it is a direct consequence of the next result where we also detail where are located the (finite) equilibrium points and how are the possible phase portraits.

\begin{theorem} \label{thm:5}
	Consider system \eqref{eq:2} or equivalently \eqref{eq:7}. There exists a period annulus at infinity if and only if we are in one of the three following cases.
\begin{enumerate}[(a)]
		\item The conditions $\gamma_L=\gamma_R=0$ and $b=0$ hold. Then, we also have $y_L=y_R=0$ and the phase plane is the result of matching two linear centers, both symmetric with respect to the $x$-axis, located at the points $(x_L,0)$ and $(x_R,0)$, which can be real or virtual equilibria. Moreover, the system is reversible and if at least one of such equilibrium points is virtual then the center is global.
		\item The conditions $\gamma_L=-\gamma_R\ne0$, $x_L=x_R=0,$ and $b=0$ hold. Then, we also have $y_L=y_R=0$, and the origin is a boundary focus from both sides, constituting a reversible global nonlinear center.
		\item The conditions $\gamma_L=-\gamma_R\ne0$, $x_L=-x_R\ne0,$ and $b=0$ hold. Then, we also have $y_L=y_R\ne0$, so that we have two real equilibria when $x_L<0<x_R$ and two virtual ones when $x_R<0<x_L$. The phase plane exhibits a reversible nonlinear center at infinity. Such a center is not global when there are real equilibria, ending in a heart-shaped homoclinic orbit to a pseudo-saddle at the origin, which contains the two foci in its interior. For the case of virtual equilibria, the origin behaves as a global nonlinear center.
	\end{enumerate}   
\end{theorem}
\begin{figure}[h]
	\includegraphics{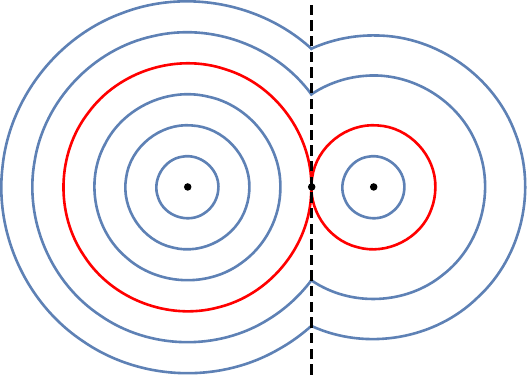} \quad
	\includegraphics{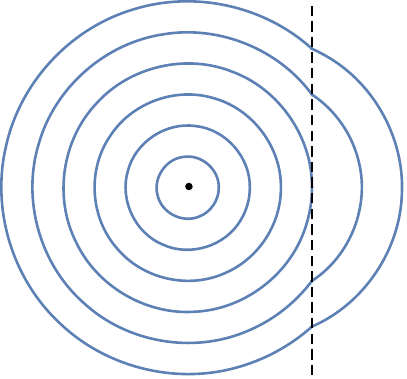} \quad
	\includegraphics{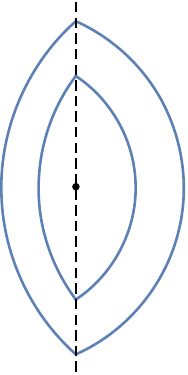}
	\caption{The centers corresponding to Theorem~\ref{thm:5}.(a): The non-global (left) and the two global ones (middle and right).}
	\label{fi:1}
\end{figure}
\begin{figure}[h]
	\includegraphics{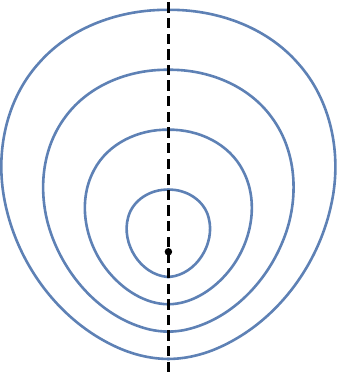} \quad
	\includegraphics{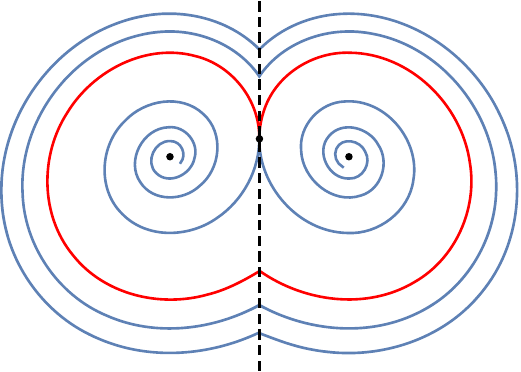} \quad
	\includegraphics{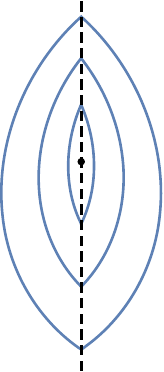}
	\caption{In the left, the global center corresponding to Theorem~\ref{thm:5}.(b); in the middle and right the centers corresponding to Theorem~\ref{thm:5}.(c), the non-global and the global one, respectively.}
	\label{fi:2}
\end{figure}  

In Figures~\ref{fi:1} and \ref{fi:2} we have drawn the phase portraits of the different centers of Theorem~\ref{thm:5}. The boundary of the period annuli when the centers are non-global are depicted in red. Clearly, the global centers have only one period annulus, while the non-global centers have either three period annuli or only one. 

\medskip

A direct application of the above result is the center classification when system \eqref{eq:2} is continuous, see Corollary~\ref{coro:2}. In this special case, only centers of type (a) or (b) appear. Clearly, centers of type (c)  are always discontinuous, since from \eqref{eq:8} we have $\alpha_L=-\alpha_R\ne 0$. More concretely, for continuous centers of type (a) the system is indeed purely linear, and so both equilibria are located at the same point $(\alpha,0)$, which becomes a global linear center, being $\alpha$ the common value for the non-homogeneous terms. Centers of type (b) are in fact always continuous yet nonlinear.

\medskip

Before proving the centers' characterization theorem, we show a simple technical result relating the first derivative at zero of the displacement map $\Delta$ in \eqref{eq:11} with the sum of the divergence of left and right systems in \eqref{eq:2}. In fact, it will characterize, when the first difference $L_1-R_1$ is non-vanishing, the stability of the periodic orbit at infinity.

\begin{lemma}\label{lem:6} Regarding \eqref{eq:15} and \eqref{eq:18}, the equality 
	\[
	\operatorname{sign} (L_1-R_1)=\operatorname{sign} (\gamma_L+\gamma_R)
	\]
	is true.
\end{lemma}
\begin{proof}
	We have 
	\[
	L_1-R_1=\e_R^{+}-\e_L^{-}=\e^{\pi\gamma_R}-\e^{-\pi\gamma_L}=\e^{-\pi\gamma_L}\left(\e^{\pi(\gamma_R+\gamma_L)}-1\right),
	\]
	and the conclusion follows easily.
\end{proof}

\begin{proof}[Proof of Theorem~\ref{thm:5}]
The first step shows that the conditions given in statements (a), (b), and (c) are sufficient for the existence of a period annulus near infinity. In the second step will see that they are also necessary. The main tool will be the study of the displacement map $\Delta(u_0)=L(u_0)-R(u_0),$ presented in Section~\ref{se:3}.

\medskip

We start the first step by assuming that we are under the conditions of statement (a). To see that these conditions assure the existence of a period annulus at infinity it suffices to consider that the system reduces to the  discontinuous zero-divergence piecewise linear system (of \emph{sewing} and \emph{refracting} type, see \cite{BuzMedTei2013,MedTor2015,PonRosVel2018})
\begin{equation*}
\begin{array}{ll}
    \left\{\begin{array}{l}
        \dot{x}=- y, \\ 
        \dot{y}=x -x_L,
    \end{array}\right. \ \text{ for } x\leq0; \quad
    &
    \left\{\begin{array}{l}
        \dot{x}=-y, \\ 
        \dot{y}=x -x_R, 
    \end{array}\right.   \text{ for } x>0. \quad
\end{array}
\end{equation*}
This piecewise system admits the time-reversibility $(x,y,\tau)\mapsto(x,-y,-\tau).$ Then, the functions $L$ and $R$ in \eqref{eq:11} satisfy $L(u_0)=-u_0$ and $R(u_0)=-u_0.$ Hence, the displacement function $\Delta$ vanishes identically for all $u_0>0$. We have so a center near infinity, resulting from the matching of two (real or virtual) linear centers. See the different phase portraits in Figure~\ref{fi:1}. 

Additionally, we have also the first integrals  $H_L(x,y)=(x-x_L)^2+y^2$ for $x<0$ and $H_R(x,y)=(x-x_R)^2+y^2$ for  $x\geq0$. Depending on the values of $x_L$ and $x_R$ we can have none, one, or two real equilibrium points surrounded by closed periodic orbits, to be either circles contained in one half-plane or the concatenation of two arcs of the form
\[
\begin{cases}
(x-x_L)^2+y^2=h_L, & \text{for } x\leq 0,\\
(x-x_R)^2+y^2=h_R, &  \text{for } x> 0,\\
\end{cases}
\]
intersecting at the two points $(0,\pm y_{h})$, with $y_{h}\geq0$, such that
\[
y_{h}^2=h_L-x_L^2=h_R-x_R^2,
\]
where the values $h_L\ge x_L^2$ and $h_R\ge x_R^2$ must satisfy the last equality.

\medskip

Considering now the conditions given in statement (b), the system becomes the continuous piecewise linear system
\[
\begin{array}{ll}
    \left\{\begin{array}{l}
        \dot{x}=2\gamma_Lx- y, \\ 
        \dot{y}=x,
    \end{array}\right. \text{ for } x\leq0; \quad
    &
    \left\{\begin{array}{l}
        \dot{x}=-2\gamma_Lx-y, \\ 
        \dot{y}=x,
    \end{array}\right.   
\end{array}  \text{ for } x\geq 0,
\]
which is well known to have a global nonlinear center at the origin (\cite{FPT12}), so that $\Delta$ vanishes identically. Note that the system admits the time-reversibility $(x,y,\tau)\mapsto(-x,y,-\tau)$, see Figure~\ref{fi:2} left.

\medskip

Regarding statement (c), the system becomes
\[
    \left\{\begin{array}{l}
        \dot{x}=2\gamma_Lx-y, \\ 
        \dot{y}=(1+\gamma_L^2)(x-x_L),
    \end{array}\right. \text{for } x\leq0; \quad
    \left\{\begin{array}{l}
        \dot{x}=-2\gamma_Lx-y, \\ 
        \dot{y}=(1+\gamma_L^2)(x+x_L),
    \end{array}\right.   
\text{for } x\geq0,
\]
which, as in the previous case, admits the reversibility $(x,y,\tau)\mapsto(-x,y,-\tau)$. Clearly, this reversibility allows us to show that any arc of orbit in one half-plane with the two endpoints on the $y$-axis determines a closed orbit, so that the existence of a period annulus at infinity is guaranteed. Excluding the cases $x_L=0$ or $\gamma_L=0$ that lead to previously studied cases, for the case with $x_L<0$ and $\gamma_L>0$ such a period annulus at infinity terminates at a bounded heart-shaped closed orbit, which behaves like a homoclinic orbit to a pseudo-saddle at the origin, formed by the collision of two visible tangencies and containing two foci of opposite stability in its interior, see Figure~\ref{fi:2} right. If $x_L<0$ and $\gamma_L<0,$ then the situation is analogous, but this time the period annulus at infinity terminates at an inverted heart-shaped closed orbit, containing the two foci. When $x_L>0$ we have just a pseudo-center at the origin, where there are two invisible tangencies. 

\medskip

As we have mentioned, the second step follows by considering the displacement function $\Delta(u_0)=L(u_0)-R(u_0)$ for $u_0>0$ defined in \eqref{eq:11}, we will have a period annulus near the periodic orbit at infinity if there exists $\varepsilon>0$ such that $\Delta(u_0)=0$ for all $0<u_0<\varepsilon$. This implies, since $\Delta$ is an analytic function at $u_0=0$, that all its derivatives should vanish at 0. First, from Lemma~\ref{lem:6} we know that
\begin{equation}\label{eq:19}
\Delta_1=L_1-R_1=0 \text{ if and only if } \gamma_L+\gamma_R=0.
\end{equation}
Assuming such a condition, that is, $\gamma_R=-\gamma_L$ and therefore $ \e_R^{+}=\e_L^{-}$, from \eqref{eq:16} and \eqref{eq:18} we have
\[
\Delta_2=L_2-R_2=\e_L^{-}(1 + \e_L^{-})(y_R-y_L),
\]
and so,
\begin{equation}\label{eq:20}
\Delta_2=L_2-R_2=0  \text{ if and only if } y_R-y_L =0.
\end{equation}

We study now the condition $\Delta_3=0,$ when $\Delta_1=\Delta_2=0.$ If we assume that \eqref{eq:19} and \eqref{eq:20} hold, then we see from \eqref{eq:17}  and \eqref{eq:18} that 
\begin{equation}\label{eq:21}
\Delta_3=L_3-R_3=\dfrac{\e_L^{-}}{2}(1-(\e_L^{-})^2)(1+\gamma_L^2)(x_L^2-x_R^2).
\end{equation}
Three possibilities arise for \eqref{eq:21} to vanish. First, we must study the case $\e_L^{-}=1$, which leads to $\gamma_L=0$ and then, from \eqref{eq:19} and  \eqref{eq:20}, we are under the conditions of statement (a). 

A second possibility for \eqref{eq:21} to vanish is the case $x_R=x_L$. Assuming again \eqref{eq:19} and \eqref{eq:20}, we get
\begin{equation}\label{eq:22}
\Delta_4=L_4-R_4=-\dfrac{4\e_L^{-}}{3}\left(1+(\e_L^{-})^3\right)(1 +  \gamma_L^2) \gamma_L x_L^3.
\end{equation}
In this case, from \eqref{eq:5} and \eqref{eq:6}, additionally we have the condition $y_L=y_R=0.$ We conclude that \eqref{eq:22} vanishes only either if $\gamma_L=0$, and then we are in the case of statement~(a), or if $x_L=0$, being then under the conditions of statement~(b), where we have again as a consequence $b=0$. 

Finally, the third possibility for \eqref{eq:21} to vanish, once assumed conditions \eqref{eq:19} and \eqref{eq:20}, is the case $x_R=-x_L$, which again implies $b=0$ and also $y_L=y_R$ (not necessarily zero, this time). In short, we are in  statement~(c). 
\end{proof}

\section{Weak-foci and its perturbations}\label{se:5}

In this section, we will prove Theorem~\ref{thm:4}. Firstly, we deal with the part concerning the hyperbolicity and stability of the periodic orbit at infinity and the possible weak-focus orders that it can have, see Theorem~\ref{thm:7}. Secondly, Theorem~\ref{thm:8} provides a complete description about when system \eqref{eq:2} exhibits $1,2,$ or $3$ limit cycles bifurcating from the different possible weak-focus orders.

\begin{theorem}\label{thm:7}
	For system \eqref{eq:2}, or equivalently for system \eqref{eq:7}, the periodic orbit at infinity is hyperbolic and stable (unstable) whenever $\gamma_L+\gamma_R>0$ $(\gamma_L+\gamma_R<0)$. In the case $\gamma_L+\gamma_R=0$ the periodic orbit at infinity is non-hyperbolic so that it behaves like a weak-focus or a center. In such a case, the following statements hold.
	\begin{enumerate}[(a)]
		\item If $\gamma_L=-\gamma_R$ and $y_L-y_R\ne0$ (equivalently, $b-\gamma_L(x_L+x_R)\ne0$), then the periodic orbit at infinity behaves like a weak-focus of order $1.$ It is stable when $y_R-y_L>0$  (equivalently, $b>\gamma_L(x_L+x_R)$) and unstable otherwise. 
		\item If $\gamma_L=-\gamma_R\ne0$ and the two conditions $y_L=y_R$ and $x_L^2-x_R^2\ne0$ hold, then the periodic orbit at infinity behaves like a weak-focus of order $2$ and it is stable (unstable) when $\gamma_L(x_L^2-x_R^2)>0$ $(\gamma_L(x_L^2-x_R^2)<0)$. 
		\item If $\gamma_L=-\gamma_R\ne0$ and the two conditions $y_L=y_R=0$ and $x_L=x_R\ne0$ hold, then the periodic orbit at infinity behaves like a weak-focus of order $3$ and it is stable (unstable) when $\gamma_Lx_L<0$ $(\gamma_Lx_L>0)$. 
		\item Otherwise, that is, when the three conditions $\gamma_L=-\gamma_R$, $y_L=y_R$ and $x_L=-x_R$ hold, so that $b=0$ also holds, we are in one of the three center cases of Theorem~\ref{thm:5}. Thus, the periodic orbit at infinity is stable but not isolated, and so it is not orbitally asymptotically stable. 
	\end{enumerate}
\end{theorem}
\begin{proof}
Following the notation used at the beginning of Section~\ref{se:3}, the periodic orbit at infinity will be stable (unstable) when, for $u_0>0$ and small, we have $y_1-y_2<0$ $(y_1-y_2>0)$. After multiplying by $u_1u_2>0$, we get that the periodic orbit at infinity will be stable (unstable) when, for $u_0>0$ and small, we have $u_2-u_1<0$ $(u_2-u_1>0)$, or equivalently $u_1-u_2=\Delta(u_0)>0$ $(\Delta(u_0)<0)$. Moreover, under this non-vanishing condition the derivative of the pseudo-return map is not the unity and the periodic orbit at infinity is hyperbolic. We notice that the computation of the derivative of the pseudo-return map has been done also in Section~\ref{se:3}. Hence, the first statement about stability when $\gamma_L+\gamma_R\ne0$ follows directly from Lemma~\ref{lem:6}. 

\medskip
When the quotient is the unity value, i.e. $\gamma_L+\gamma_R=0$, we are in the non-hyperbolic case. Then, the assertions on the stability require to consider higher-order derivatives of the displacement function $\Delta$ at $u_0=0$, which allow to determine the sign of $\Delta(u_0)$ for small $u_0>0$. Statements (a), (b), and (c) come from the expressions \eqref{eq:20}, \eqref{eq:21}, and \eqref{eq:22}, respectively.  Statement (d) is a direct consequence of Theorem~\ref{thm:5}.
\end{proof}

Note that from Corollary~\ref{coro:2} it is easy to check that statements (b) and (c) in Theorem~\ref{thm:7} actually correspond to discontinuous systems \eqref{eq:2}. Effectively, we have then $\gamma_L=-\gamma_R$ so that the necessary condition for continuity $\alpha_L=\alpha_R$ fails in (b) since $x_L\ne x_R$. Although $\alpha_L=\alpha_R$ in case (c), this time the condition $b=0$ is not fulfilled, since then $b=2\gamma_L x_L$. Hence, for continuous systems \eqref{eq:2} the periodic orbit at infinity can only behave like a weak-focus of order $1$.

\medskip

Next result proves the second statement of Theorem~\ref{thm:4}. 

\begin{theorem} \label{thm:8}
	System \eqref{eq:2}, or equivalently system \eqref{eq:7}, undergoes a degenerated Hopf bifurcation at infinity for $\gamma_L+\gamma_R=0$, and the following statements hold.
	\begin{enumerate}[(a)]
		\item If we take $\gamma_R$ as the only bifurcation parameter, assuming fixed values for the remaining parameters, and the condition $y_L-y_R\ne0$ (equivalently, $b-\gamma_L(x_L+x_R)\ne0)$ holds, then one hyperbolic stable (unstable) limit cycle bifurcates from infinity for $\gamma_L<-\gamma_R$  $(\gamma_L>-\gamma_R)$ provided that $y_L-y_R<0$ $(y_L-y_R>0)$.
		\item If we take $(\gamma_R,b)$ as bifurcation parameters, assuming fixed values for the remaining parameters, and the condition $x_L^2-x_R^2\ne0$ holds, then the critical point $(\gamma_R,b)=(-\gamma_L,\gamma_L(x_L+x_R))$ is a bifurcation point of co-dimension two. Consequently, up to $2$ limit cycles can bifurcate from infinity.
		\item If we assume fixed values for $\gamma_L\ne0$ and $x_L\ne0$, then within the three-parameter space $(\gamma_R,b,x_R)$  the critical point $(\gamma_R,b,x_R)=(-\gamma_L,2\gamma_Lx_L,x_L)$ is a bifurcation point of co-dimension three. In particular, up to three limit cycles can bifurcate from infinity, so that in a neighborhood of such a critical point there are parameter values for which the system exhibits $3$ limit cycles of big amplitude. 
	\end{enumerate}
\end{theorem}

It should be noticed that statement (a) of Theorem \ref{thm:8} is the only that could apply to continuous systems \eqref{eq:2},  providing the bifurcation of a unique limit cycle from the corresponding weak focus of order $1.$ Recall, as we have explained in the introduction, that such systems can exhibit at most one limit cycle.

The most degenerate case comes from the situation described in statement (c) of Theorem~\ref{thm:7}, when the periodic orbit at infinity behaves like a weak-focus of order three. In such a case, we have $\gamma_L=-\gamma_R\ne0$, $y_L=y_R,$ and $x_L=x_R\ne0$, so that from \eqref{eq:8}, we have $b=2\gamma_L x_L$. We will omit the proof of statements (a) and (b), paying only attention to the more involved statement (c). In fact the existence of a Hopf bifurcation is clear from \eqref{eq:20} and the linearity condition \eqref{eq:19}. We know from Theorem~\ref{thm:7} that the maximal degeneration of the periodic orbit at infinity arises when we are in the situation of statement (c). Thus, we can assume that the parameters $\gamma_L\ne0$ and $x_L\ne0$ are fixed, while we have at our disposal the three remaining parameters $\gamma_R,$ $b,$ and $x_R$. For the critical situation when $\gamma_L=-\gamma_R$, $y_L=y_R,$ and $x_L=x_R$ we know that the periodic orbit at infinity behaves like a weak-focus of order 3. Note that then condition \eqref{eq:8} reads $b=2\gamma_Lx_L-y_L=2\gamma_Lx_L+y_L$, so that we have indeed $y_L=y_R=0$; therefore, the critical value for $b$ is $2\gamma_Lx_L\ne0$. In short, we can state the following result that allows us to complete the proof of Theorem~\ref{thm:8}.

\begin{lemma}\label{lem:9} Consider system \eqref{eq:2}, or equivalently system \eqref{eq:7},  for  $\gamma_L\ne0$ and $x_L\ne0$ fixed and the three remaining parameters $\gamma_R$, $b,$ and $x_R$ in a neighborhood of the critical point $(\gamma_R,b,x_R)=(-\gamma_L,2\gamma_Lx_L,x_L)$, where the periodic orbit at infinity behaves like a weak-focus of order $3$, so that the coefficients $\Delta_i(\gamma_R,b,x_R)$ satisfy 
\[
 \Delta_i(-\gamma_L,2\gamma_Lx_L,x_L)=0,
\]
for $i=1,2,3,$ while
\begin{equation}\label{eq:23}
 \Delta_4(-\gamma_L,2\gamma_Lx_L,x_L)=-\dfrac{4\e_L^{-}}{3}\left(1+(\e_L^{-})^3\right)(1 +  \gamma_L^2) \gamma_L x_L^3\ne0.
\end{equation}
Furthermore, there exist values for $(\gamma_R,b,x_R)$ in such a neighborhood where the system has $3$ hyperbolic periodic orbits of big amplitude.
\end{lemma}
\begin{proof}
We start by computing the derivatives of the coefficients $(\Delta_1,\Delta_2,\Delta_3)$ with respect to the parameters $(\gamma_R,b,x_R)$ at the critical point. Clearly, we have
\[\frac{\partial \Delta_1}{\partial \gamma_R}=\pi \e_L^{-},
\quad \frac{\partial \Delta_1}{\partial b}=\frac{\partial \Delta_1}{\partial x_R}=0,
\]
so that we only need to compute at the critical point the following Jacobian matrix
\[
\frac{\partial (\Delta_2,\Delta_3)}{\partial (b,x_R)}=
\left(\begin{array}{cc}
2 \e_L^{-} (1 + \e_L^{-}) & -2 \e_L^{-} (1 + \e_L^{-})\gamma_L \\
0 & -\e_L^{-}  (1 -(\e_L^{-} )^2) (1 + \gamma_L^2) x_L 
\end{array}\right).
\]
Therefore, we obtain that for the Jacobian matrix at the critical point $(\gamma_R,b,x_R)=(-\gamma_L,2\gamma_Lx_L,x_L)$, we have 
\begin{equation*}
	\det\left( \frac{\partial (\Delta_1,\Delta_2,\Delta_3)}{\partial (\gamma_R,b,x_R)}\right)=-2\pi (\e_L^{-})^3(1 - \e_L^{-})(1 + \e_L^{-} )^2 (1 + \gamma_L^2) x_L \ne0.
\end{equation*}
So that there exists, in the working neighborhood, a one-to-one correspondence between the $\Delta_i$-values and the three free parameter values. 

The proof finishes using the Implicit Function Theorem and the Weierstrass Preparation Theorem, that allow us to take new local coordinates $\delta=(\delta_1,\delta_2,\delta_3)$ around zero so that, because $\Delta_4\ne0 $ in \eqref{eq:23}, we have
\begin{equation}\label{eq:24}
\Delta(u,\delta)=q(u,\delta)\widetilde{\Delta}(u,\delta)=(\delta_1 u+\delta_2 u^2+\delta_3 u^3+u^4)\widetilde{\Delta}(u,\delta),
\end{equation}
where $\widetilde{\Delta}$ is an analytical non-vanishing function at $(0,0).$ It is also clear that, under these conditions, there will be no more than three positive zeros. 
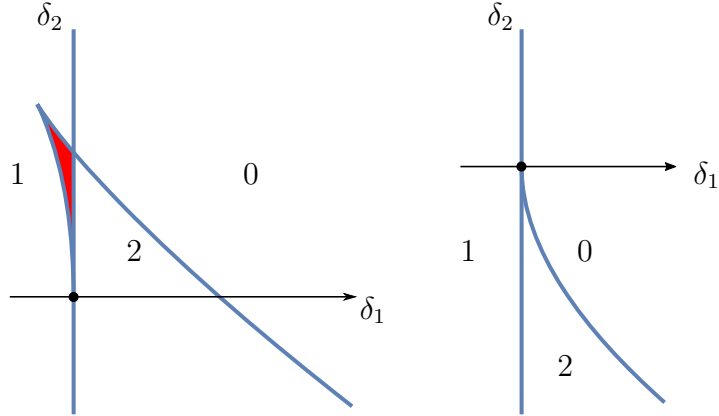
\begin{figure}[h]
	\begin{overpic}{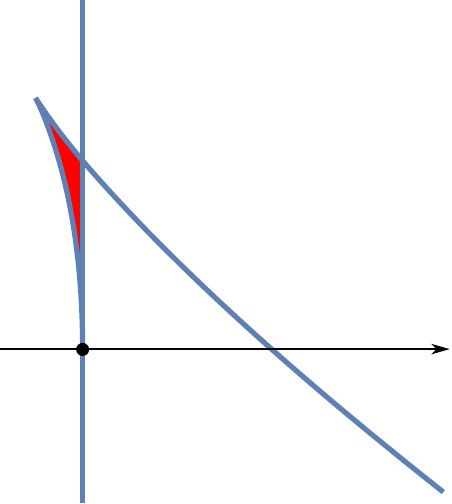}
		\put(90,25){$\delta_1$}
		\put(7,101){$\delta_2$}
		\put(60,60){$0$}
		\put(0,60){$1$}
		\put(30,40){$2$}
	\end{overpic}
	\quad \quad \quad
	\begin{overpic}{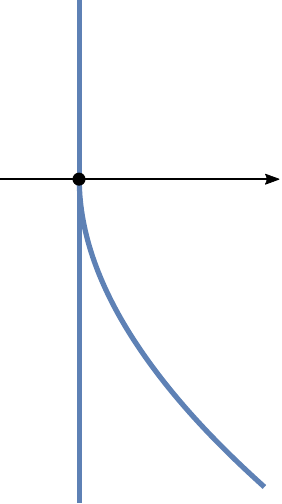}
		\put(60,60){$\delta_1$}
		\put(7,101){$\delta_2$}
		\put(30,40){$0$}
		\put(0,40){$1$}
		\put(25,10){$2$}
	\end{overpic}
	\caption{The number of positive solutions for $q(u,\delta)$ in the plane $(\delta_1,\delta_2)$ for $\delta_3<0$ (left) and $\delta_3\ge0$ (right). The region in red corresponds to the existence of $3$ positive solutions.}
	\label{fi:3}
\end{figure}  
The bifurcation curves shown in Figure~\ref{fi:3} follow directly studying when the discriminant of $q$ with respect to $u$ vanishes, namely on the varieties $\delta_1=0$ and $4\delta_1\delta_3^3-\delta_2^2\delta_3^2-18\delta_1\delta_2\delta_3+4\delta_2^3+2\delta_1^2=0.$ The cusp point in Figure~\ref{fi:3} (left) is located at $(\delta_1,\delta_2)=(\delta_3^3/27, \delta_3^2/3).$ This explains the small size of the parameters region where 3 positive zeros exist for $\delta_3<0$. Clearly, by the Descartes' rule there cannot be three positive zeros when $\delta_3>0.$
\end{proof}

We remark that, as an alternative approach to assure that the bifurcation in the above proof is determined with only the first four terms of the Taylor series of $\Delta$ whenever $\Delta_4\ne0$, we can take advantage of the ${\mathcal Z}_2$-Classification theorem in \cite{GS84}. Effectively, we can write
the function $q$ in \eqref{eq:24}, naming $u=v^2$ and removing a factor $v$, as
\[
\hat{q}(v)=\delta_1 v+\delta_2 v^3+\delta_3 v^5+v^7,
\]
which corresponds with case 5 of Figure 6.2 in \cite[page 269]{GS84}. As we have mentioned before, this bifurcation is by no means different from a degenerated Hopf bifurcation for a monodromic non-degenerated equilibrium point, i.e. those for which its Jacobian matrix has zero trace and positive determinant.

\section{Limit cycles bifurcating from the centers}\label{se:6}
In a fixed class of systems having centers, the maximum number of limit cycles that can bifurcate from a center is known as the \emph{local cyclicity} of that center. This is a very difficult problem and it is solved for a very few classes of differential systems. The analytic quadratic class is one of them and it was proved 70 years ago by Bautin that at most three limit cycles of small amplitude can bifurcate, see \cite{Bau1954}. Another instance is the class of cubic systems without quadratic nonlinearities, which was studied by Sibirski\u{\i} in \cite{Sib1965}, but the complete proof that only $5$ limit cycles of small amplitude bifurcate at the origin was done 30 years later by {\.Z}o{\l}{\c a}dek in \cite{Zol1994a}. Up to our knowledge, this question has not solved for other general classes, either for the complete cubic polynomial class. 

\medskip

In the non-smooth differential systems world, the existence of sliding segments, as is the case when $b$ is not zero in our main system \eqref{eq:2}, makes this problem very intricate because the return map is not analytic. This is not the case however when the return map is studied near the periodic orbit of infinity for system \eqref{eq:2}, as we have seen in Section~\ref{se:3}. Typically, the main technique used to bound the local cyclicity is the study of the Bautin ideal formed by the coefficients of the difference map $\Delta$ as we have defined in \eqref{eq:4} for our study. These coefficients are known as the \emph{Liapunov quantities} associated to the center-focus problem of a planar differential system, see more details in \cite{AndLeoGorMai1973}. We recall that in this context the Liapunov quantities are polynomials in the perturbation parameters, see \cite{CimGasManMan1997}. The finiteness property is proved usually using that the Bautin ideal is defined in a Noetherian ring because the number of parameters is finite, and that such ideal is radical. In our context, the coefficients $\Delta_i$ are not polynomials in the perturbation parameters, as we have already seen in Section~\ref{se:4}. 

\medskip

Due to the difficulties to study upper bounds, in this section, we deal with the study of lower bounds for the maximum number of limit cycles of big amplitude bifurcating from the centers presented in the classification Theorem~\ref{thm:3}. We will see that, for some of them, see Propositions~\ref{prop:11} and \ref{prop:12}, a transversal weak-foci curve of order $3$ emerges at the critical point in the parameters space corresponding to a center configuration. The transversality ensures again the existence of three limit cycles in a neighborhood of such a curve. A higher-order analysis is required and we prove that the local cyclicity changes when moving the parameters in the selected center family. That is, although generically the cyclicity of a family of centers remains constant, over some singular locus it can increase. We will closely follow the scheme of \cite{GinGouTor2020}.

\medskip

We start perturbing center family $(c)$ in Theorem~\ref{thm:5} because of its simplicity. The necessary computations for the other two are more involved. For simplicity, we will use the equivalent canonical form \eqref{eq:7} instead of \eqref{eq:2}.

\begin{proposition}\label{prop:10} Let  $\eta,\xi$ be non-zero real numbers. The local cyclicity of the periodic orbit at infinity of center type defined by $\gamma_L=-\gamma_R=- \eta,$ $x_L=-x_R=-\xi,$ and $b=0,$ when perturbed in the class of systems \eqref{eq:7}, is at least $2.$
\end{proposition}
\begin{proof}
	We start taking the perturbed system~\eqref{eq:7} being 
	\[(\gamma_L,\gamma_R,b,x_L,x_R)=(-\eta+\varepsilon_1,\eta+\varepsilon_2,\varepsilon_3,-\xi+\varepsilon_4,\xi+\varepsilon_5)
	\]	
	and computing the coefficients of $\Delta_i(\eta,\xi;\varepsilon),$ for $i=1,\ldots,4,$ defined in Section~\ref{se:4},  being $\varepsilon= (\varepsilon_1,\ldots,\varepsilon_5).$ The second step is the computation of their first-order expansions $\Delta^{1}_i(\eta,\xi;\varepsilon)$ for the Taylor series of $\Delta_i$ with respect to $\varepsilon= (\varepsilon_1,\ldots,\varepsilon_5).$  We can easily check that the matrix (of size $3\times 5$) defined by the coefficients of $(\Delta_1^{1},\Delta_2^{1},\Delta_3^{1})$ with respect to $\varepsilon$ has rank $3$ whenever the parameters $\eta,\xi$ are non-vanishing. In fact, the Jacobian matrix of $(\Delta_1^{1},\Delta_2^{1},\Delta_3^{1})$ with respect to $(\varepsilon_1,\varepsilon_3,\varepsilon_4)$ has a determinant 
	\[
	2\pi\xi (\e^{\eta\pi})^3 (\e^{\eta\pi}-1)(\e^{\eta\pi}+1)^2(\eta^2+1)\ne 0.
	\]
	We notice that the rank of the matrix (of size $4\times 5$) defined by the coefficients of $(\Delta_1^{1}, \Delta_2^{1}, \Delta_3^{1},\Delta_4^{1})$ with respect to $\varepsilon$ remains unchanged, being also $3$.	
	Hence, using the Implicit Function Theorem there exist new local coordinates $\varepsilon_\delta=(\delta_1,\varepsilon_2,\delta_2,\delta_3,\varepsilon_5)$ in a neighborhood of the origin in the parameters space, such that $\Delta_i(\eta,\xi;\varepsilon_\delta)=\delta_i$ for $i=1,2,3.$ The proof finishes using the same argument as in Section~\ref{se:5} because we have a transversal curve of weak-foci of order $2$ bifurcating from each center value $(\eta,\xi)$, in the $2$-dimensional manifold in the parameters space. Moreover, the transversality assures the bifurcation of up to $2$ limit cycles of big amplitude. 
\end{proof}

In what follows we extend the notation $\Delta^{j}_i(\varepsilon)$ for the $j$-th order truncation of the Taylor series of $\Delta_i$ with respect to $\varepsilon.$ We notice that in the above proposition we have not get more limit cycles using $\Delta_4,$ even arriving up to fourth-order. Indeed, $\Delta_4^{4}$ vanishes when $(\delta_1,\varepsilon_2,\delta_2,\delta_3,\varepsilon_5)=(0,\varepsilon_2,0,0,\varepsilon_5).$

\smallskip

Next result provides the number of limit cycles that can bifurcate from center family $(a)$ in Theorem~\ref{thm:5}. It gives the bifurcation diagram of the number of limit cycles in a 2-dimensional manifold.

\begin{proposition}\label{prop:11}
Let  $\eta,\xi$ be non-zero real numbers. The local cyclicity of the periodic orbit at infinity of center type defined by $\gamma_L=\gamma_R=0,$ $x_L=\eta$, $x_R=\xi,$ and $b=0$, when perturbed in the class of systems \eqref{eq:7}, is at least $1$ when $\eta=-\xi,$ at least $2$ when $\eta\ne\pm\xi,$ and at least $3$ when $\eta=\xi.$
\end{proposition}
\begin{proof}
We consider a perturbation in system~\eqref{eq:7} with
\[(\gamma_L,\gamma_R,b,x_L,x_R)=(\varepsilon_1,\varepsilon_2,\varepsilon_3,\eta +\varepsilon_4,\xi+\varepsilon_5).
\]	
The Taylor series of the coefficients of $\Delta_i(\eta,\xi;\varepsilon),$ defined in Section~\ref{se:4}, with respect to $\varepsilon= (\varepsilon_1,\ldots,\varepsilon_5)$ write as
\begin{equation}\label{eq:25}
\begin{aligned}
\Delta_1(\eta,\xi;\varepsilon)&=\pi \varepsilon_1+\pi \varepsilon_2+O_2(\varepsilon),\\
\Delta_2(\eta,\xi;\varepsilon)&=-4 \eta \varepsilon_1+4\xi \varepsilon_2+4 \varepsilon_3+O_2(\varepsilon),\\
\Delta_3(\eta,\xi;\varepsilon)&=\pi \eta^2 \varepsilon_1+\pi\xi^2 \varepsilon_2+O_2(\varepsilon),\\
\Delta_4(\eta,\xi;\varepsilon)&=-\frac{4}{3} \eta^3 \varepsilon_1+\frac{4}{3}\xi^3 \varepsilon_2+O_2(\varepsilon).\\
\end{aligned}
\end{equation}
The matrix of the first three linear terms with respect to $(\varepsilon_1,\varepsilon_2,\varepsilon_3)$ has a determinant $4\pi^2(\eta^2-\xi^2).$ As it is non-zero when $\eta\ne\pm\xi,$ reasoning as in the proof of Proposition~\ref{prop:10}, it is clear that there exists a change of variables in the parameters space such that $\Delta_i(\eta,\xi;\varepsilon)=\delta_i,$ for $i=1,2,3,$ and that the local cyclicity is at least $2.$ Straightforward computations show that $\Delta_4^2=0$ when $\delta_1=\delta_2=\delta_3=0$ and we can not get more limit cycles up to a second order analysis.
\smallskip

When $\eta=-\xi\ne0,$  the first two linear Taylor series in \eqref{eq:25} are linearly independent and the rank of the corresponding matrix adding the next two rows does not increase. Using again the Implicit Function Theorem, we can use new local coordinates $(\delta_1,\varepsilon_2,\delta_2,\varepsilon_4,\varepsilon_5)$ in a neighborhood of the origin so that $\Delta_i=\delta_i,$ for $i=1,2.$  Straightforward computations show that $\Delta_3^2=\Delta_4^2=0$ when $\delta_1=\delta_2=0$ and we can not get more limit cycles up to a second order analysis.

\smallskip

Finally, when $\eta=\xi\ne0$ we need to work with Taylor series of second order. Doing as above and using again the Implicit Function Theorem, we can take new local coordinates $(\delta_1,\delta_2,\varepsilon_3,\varepsilon_4,\varepsilon_5)$ such that $\Delta_i=\delta_i,$ for $i=1,2.$ Taking $\delta_1=\delta_2=0,$ we have that \eqref{eq:25} reduces to
\begin{equation}\label{eq:26}
\begin{aligned}
\Delta_3(\xi;\hat\varepsilon)&=\pi\varepsilon_3(\varepsilon_4-\varepsilon_5)+O_3(\hat \varepsilon),\\
\Delta_4(\xi;\hat\varepsilon)&=-\frac{4}{3}\xi^2\varepsilon_3+O_2(\hat\varepsilon),\\
\end{aligned}
\end{equation}
with $\hat{\varepsilon}=(\varepsilon_3,\varepsilon_4,\varepsilon_5).$ From now on, we can simplify the computations taking $\varepsilon_5=0.$ Then, doing a blow-up change of coordinates $(\varepsilon_3,\varepsilon_4)=(\varepsilon_3,\tilde \varepsilon_4\varepsilon_3),$ equation \eqref{eq:26} writes as
\begin{equation*}
\begin{aligned}
\Delta_3(\xi;\varepsilon_3,\tilde\varepsilon_4)&=\varepsilon_3^2(\pi \tilde \varepsilon_4+\varepsilon_3\,O_0(\varepsilon_3,\tilde \varepsilon_4)),\\
\Delta_4(\xi;\varepsilon_3,\tilde\varepsilon_4)&=\varepsilon_3(-\frac{4}{3}\xi^2+\varepsilon_3\,O_0(\varepsilon_3,\tilde \varepsilon_4)).\\
\end{aligned}
\end{equation*}
The Implicit Function Theorem allows us to define a new local coordinate $\delta_3$ so that $\Delta_3(\xi;\varepsilon_3,\tilde\varepsilon_4)=\varepsilon_3^2\delta_3.$  The proof finishes by imposing that the new coordinate $\delta_3$ to be zero and checking that when $\varepsilon_3$ is small but not zero, the fourth coefficient $\Delta_4$ is non-vanishing because $\xi\ne0$. Consequently, we have a third-order weak-focus curve that is born from the critical parameter values corresponding to the center, from which the $3$ limit cycles can bifurcate.
\end{proof}

The strategy used at the end of the last proof is an interesting non-standard use of the Implicit Function Theorem in this field, even it frequently employed in desingularization procedures in singularity theory. Up to the best of our knowledge, this procedure goes back to Loud in \cite{Lou1961}.

\smallskip

We finish this section perturbing the remaining family (b) in Theorem~\ref{thm:5}.
\begin{proposition}\label{prop:12}
Let  $\eta$ be a non-zero real number. The local cyclicity of the periodic orbit at infinity of center type defined by $\gamma_L=-\gamma_R=-\eta,$ $x_L=x_R=0,$ and $b=0$, when perturbed in the class of systems \eqref{eq:7}, is at least $3.$
\end{proposition}
\begin{proof}
As in the previous two proofs, we consider a general perturbation in system~\eqref{eq:7} with
\[
(\gamma_L,\gamma_R,b,x_L,x_R)=(-\eta+\varepsilon_1,\eta+\varepsilon_2,\varepsilon_3,\varepsilon_4,\varepsilon_5).
\]	
The Taylor series of the coefficients of $\Delta_i(\eta,\xi;\varepsilon),$ defined in Section~\ref{se:4}, with respect to $\varepsilon= (\varepsilon_1,\ldots,\varepsilon_5)$ write as
\[
\begin{aligned}
\Delta_1(\eta;\varepsilon)&=\e^{\eta\pi}\pi(\varepsilon_1+\varepsilon_2)+O_2(\varepsilon),\\
\Delta_2(\eta;\varepsilon)&=2\e^{\eta\pi}(\e^{\eta\pi}+1)(\varepsilon_3+\eta(\varepsilon_4+\varepsilon_5))+O_2(\varepsilon),\\
\Delta_3(\eta;\varepsilon)&=O_2(\varepsilon),\\
\Delta_4(\eta;\varepsilon)&=O_2(\varepsilon).
\end{aligned}
\]
As before we can use Taylor series of order $1$ and the Implicit Function Theorem to define new local coordinates $(\delta_1,\varepsilon_2,\delta_2,\varepsilon_4,\varepsilon_5)$ so that $\Delta_1=\delta_1$ and $\Delta_2=\delta_2.$ Restricting our attention to the manifold $\delta_1=\delta_2=0,$ the next two coefficients write, up to order $3$ terms, as
\[
\begin{aligned}
\Delta_3(\eta;\tilde\varepsilon)&=-2\tilde\varepsilon_4\tilde\varepsilon_5(\eta^2+1)(\e^{2\eta\pi}-1)\e^{\eta\pi} +O_4(\tilde\varepsilon),\\
\Delta_4(\eta;\tilde\varepsilon)&=\frac{4}{3}\eta(\eta^2+1)\e^{\eta\pi}\tilde\varepsilon_4\big((\e^{3\eta\pi}+1)\tilde\varepsilon_4^2+(\e^{\eta\pi}+1)(9\e^{2\eta\pi}-6)\tilde\varepsilon_5^2\big)+O_4(\tilde\varepsilon),
\end{aligned}
\]
where $\tilde\varepsilon=(\tilde\varepsilon_4,\tilde\varepsilon_5)$ and we have taken $\varepsilon_2=0, \varepsilon_4 = \tilde\varepsilon_4+\tilde\varepsilon_5,$ and $\varepsilon_5 = \tilde\varepsilon_4-\tilde\varepsilon_5$ to simplify the above expressions. In the plane $(\tilde\varepsilon_4,\tilde\varepsilon_5)$ the curve $\Delta_3=0$ has, near the origin, two branches, one tangent to $\tilde\varepsilon_4=0$ and another to $\tilde\varepsilon_5=0.$ As the above Taylor series vanish over the first one, we should work with the second one, where $\Delta_3$ vanishes but not $\Delta_4,$ being $\tilde\varepsilon_4$ small but not zero. The proof finishes doing a new local change of coordinates of blow-up type $(\tilde\varepsilon_4,\tilde\varepsilon_5)=(\tilde\varepsilon_4,\tilde\varepsilon_4 \delta_3).$ Clearly, we have again for every nonzero $\eta$ a transversal curve of third-order weak-foci on the parameters space that is born at the critical value where system has a center and from which $3$ limit cycles can bifurcate.
\end{proof}

Note that in the proofs of this section, we could have computed even higher-order Taylor series in looking for higher cyclicity but our goal was to get good lower bounds without big computational effort.

\section{An explicit example}\label{se:7}

In the previous two sections, we have seen that the maximum number of limit cycles found for system~\eqref{eq:7} can bifurcate both from the weak-foci of maximal order and also from some of center families.

We finish the work with an explicit numerical example showing the existence of $3$ limit cycles of big amplitude. As the bifurcation near the centers is more degenerate, we will deal around a weak-focus of maximal order. Let us take for system \eqref{eq:7} the parameter values 
\begin{equation}\label{eq:27}
\gamma_L=-\gamma_R=-\frac18,\quad b=-\frac14,\quad x_L=x_R=1,
\end{equation}
so that from \eqref{eq:19}-\eqref{eq:22} we have $\Delta_1=\Delta_2=\Delta_3=0$, and
\[
\Delta_4=\frac{65}{384} \e^{\pi /8} \left(1+\e^{3\pi /8}\right)\approx 1.06495899308488,
\]
where from \eqref{eq:8} we have $y_L=y_R=0$. The phase portrait is depicted in Figure~\ref{fi:4}. 
\begin{figure}[h]
    \includegraphics{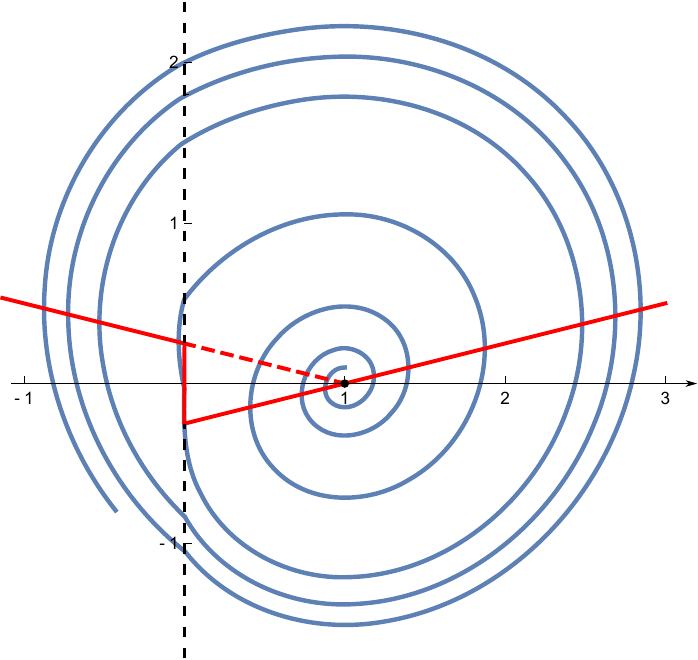} 
	\caption{Phase portrait of system~\eqref{eq:7} with the parameters values given in \eqref{eq:27}, having a periodic orbit at infinity that behaves like a weak-focus of order three. We draw in red the vertical nullclines and the sliding segment.}
	\label{fi:4}
\end{figure}  

Clearly, the chosen unperturbed system has both the virtual focus and the real one located at the same point $(1,0)$. Obviously, the perturbations providing the three limit cycles will separate them. Curiously, the first found example with three limit cycles in the family of piecewise linear differential systems separated by a straight line, which was numerically presented in \cite{HuanYang12} and later justified in \cite{LP12}, exhibited a configuration rather far from the weak-focus of order three but with the two foci located at the same point. However, such a pioneering example was not written in any canonical form; if one writes it in our Li\'enard canonical form \eqref{eq:7} then the two foci become not at the same point and appreciably distant one another.

\medskip

Coming back to our example, we consider the parametric family of perturbed systems \eqref{eq:7} with $\gamma_L=-\frac18$, $x_L=1$, and
\[
\begin{aligned}	
\gamma_R&=\frac{1}{8}+\frac{\e^{-\pi /8}}{\pi }\varepsilon_1,\\
b&=-\frac{1}{4}+\frac{1024-65\pi-(1024-195\pi)\e^{\pi /4}}{1040\pi(\e^{\pi/4}-1) \e^{\pi /8}}\varepsilon_1
+\frac{\e^{-\pi /8}}{2( \e^{\pi /8}+1)}\varepsilon_2
-\frac{8 \e^{-\pi /8}}{65( \e^{\pi /4}-1)}\varepsilon_3,\\
x_R&=1+\frac{16+65\pi-(16+195\pi)\e^{\pi /4}}{130\pi(\e^{\pi/4}-1) \e^{\pi /8}}\varepsilon_1
+\frac{64 \e^{-\pi /8}}{65 (\e^{\pi /4}-1)}\varepsilon_3,
\end{aligned}
\]			
so that we have $\Delta_i(\gamma_R,b,x_R)=\varepsilon_i+O(\varepsilon^2)$ for $i=1,2,3$, where $O(\varepsilon^2)$ represents higher-order terms in  $\varepsilon=(\varepsilon_1,\varepsilon_2,\varepsilon_3)$, obtaining a non-vanishing perturbed value for $\Delta_4$. Thus, we have a complete unfolding in the $3$-parameter space $(\gamma_R,b,x_R)$ for a neighborhood of the critical point $(\gamma_R,b,x_R)=(1/8,-1/4,1)$, which represents a weak-focus of order 3 for the periodic orbit at infinity.

In particular, for the concrete perturbed system with 
\[
\gamma_L=-\frac18, \quad x_L=1, \quad \gamma_R=\frac{1638355}{13106841},\quad b=-\frac{260534}{1045519},\quad x_R=
\frac{552751}{556327},
\]
we get
\[
y_L=-\frac{3383}{4182076}\approx -0.000808928,\quad y_R=-\frac{6084083513535748}{7623599948859945633}\approx -0.000798059.
\]
After using the relations \eqref{eq:5} and \eqref{eq:6} to get $\alpha_R$ and $\alpha_L$, the fourth-degree truncation of function $\Delta(u_0)$ becomes 
\[-4.43719886 \cdot 10^{-8} u_0+3.993655760\cdot 10^{-5} u_0^2-1.15001344\cdot 10^{-2} u_0^3+1.054869499 u_0^4,
\]
which has three simple positive zeros at 
\[
\{0.002467460261, 0.003358360933, 0.005076128658\}
\]
with reciprocal values (in reverse order)
\[
\{197.00052293,297.76430224,405.27501730\}.
\]
Accordingly,  system \eqref{eq:7} with such perturbed parameter values has three limit cycles, whose intersection points $(0,y_i)$ with the positive $y$-axis have the $y_i$ ordinates
\[
\{196.89979358,297.91820638,405.21567427\},
\]
which are very close to the reciprocals of zeros for the fourth-order truncated function for $\Delta$, whose graph is drawn in Figure~\ref{fi:5}.
\begin{figure}[h]
	\includegraphics{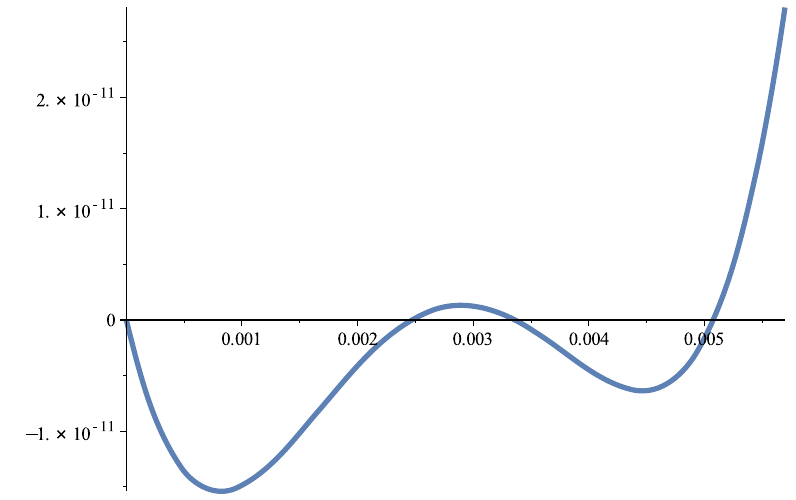}
	\caption{The graph of the fourth-degree truncation of the function $\Delta(u_0)$ for the numerical example obtained after perturbing the weak-focus of order 3 given in \eqref{eq:27}.}
	\label{fi:5}
\end{figure}  

As a final remark, it should be emphasized that only thanks to the theoretical analysis developed in this work it has been possible to detect the above example. Of course, many other analogous examples could now be built without extra effort.

\section{Acknowledgements}
This work has been realized thanks to Consejer\'{\i}a de Econom\'{\i}a y Conocimiento de la Junta de Andaluc\'{\i}a (P12-FQM-1658 grant), Ag\`encia de Gesti\'o d'Ajuts Universitaris i de Recerca de Catalunya (2017 SGR 1617 grant), Spanish Ministerio de Ci\'encia, Innovaci\'on y Universidades - Agencia Estatal de Investigaci\'on (MTM2016-77278-P (FEDER), MTM2017-87915-D2-1-P, PGC2018-096265-B-I00, and PID2019-104658GB-I00 grants), and European Community (H2020-MSCA-RISE-2017-777911 grant).

\def\cprime{$'$}

\end{document}